\documentclass[11pt, reqno]{amsart}
\usepackage{lmodern}

\usepackage[utf8]{inputenc}
\usepackage[english]{babel}
\usepackage{amsmath, amsthm, amssymb}
\usepackage{mathtools}
\usepackage{mathrsfs}
\usepackage[mathscr]{euscript}
\usepackage{stmaryrd}
\usepackage{enumerate}
\usepackage[all]{xy}
\usepackage{rotating}
\usepackage{extarrows}
\usepackage{accents}
\usepackage{leftidx}
\usepackage{upgreek}
\usepackage{bbm}
\usepackage{bbold}
\usepackage{xr-hyper}
\usepackage{tikz}
\usetikzlibrary{matrix,arrows}
\usepackage{tikz-cd}
\usepackage{hyperref}
\usepackage{csquotes}
\MakeOuterQuote{"}

\usepackage{macros}

\hypersetup{
 pdftitle = {Quasi-symmetric functions and the stack of expanded pairs},
 pdfauthor = {Jakob Oesinghaus},
 colorlinks=true}
\newcommand{\Chow}{\mathrm{CH}}
\newcommand{\Chowbullet}{\mathrm{CH}^{\bullet}}

\title[Quasi-symmetric functions and the stack of expanded pairs]{Quasi-symmetric functions and the Chow ring of the stack of expanded pairs}
\author{Jakob Oesinghaus}

\begin{document}
\begin{abstract}
We show that the Hopf algebra of quasi-symmetric functions arises naturally as the integral Chow ring of the algebraic stack of expanded pairs originally described by J. Li, using a more combinatorial description in terms of configurations of line bundles. In particular, we exhibit a gluing map which gives rise to the comultiplication. We then apply the result to calculate the Chow rings of certain stacks of semistable curves.
\end{abstract}
\maketitle
\tableofcontents

\section{Introduction}
In order to study relative Gromov-Witten invariants of a pair $(X,D)$ of a
scheme $X$ over $\bC$ and a Cartier divisor $D \subset \bC$, J. Li
introduced the notion of a family of \emph{expansions} of the pair $(X,D)$
(cf. \cite{LiStableMorphisms, LiDegeneration}).
An expansion of the pair $(X,D)$ of length $\ell$ is constructed by 
gluing $\ell$ copies of the projectivized normal bundle of $D$ to $X$:
\[
X(\ell) := X \sqcup_D \underbrace{ P \sqcup_D \cdots \sqcup_D P}_{\ell\text{ times}},
\]
where $P = \bP(\cO_X (D)\vert_D \oplus \cO_D)$. Since $P$ carries a $\bG_m$-action
by scaling, an expansion of length $\ell$ comes with an action of $\bG_m^\ell$.

Let $\cA := [\bA / \bG_m]$ be the stack quotient of the affine line by the
standard action of the multiplicative group; it is the moduli space of pairs
$(\cL, s)$ of a line bundle $\cL$ and a section $s$ of $\cL$. Let 
$\cD := [ 0 / \bG_m] \cong \bG_m$ be the vanishing locus of the universal
section. The pair $(\cA, \cD)$ forms the universal pair of an algebraic stack and a Cartier divisor on it.
In \cite{ExpDegPairs}, the authors used this fact to define a stack $\cT$ of
expansions of any expansions of the universal pair $(\cA, \cD)$,
proved that it is an algebraic stack locally of finite type. The same object
has also been studied in \cite{GraberVakilRelative}.

The stack $\cA$ and its powers $\cA^n \cong [ \bA^n / \bG_m^n]$ are connected 
to logarithmic geometry; in fact, they form an open substack of the stack of logarithmic
structures (\cite{OlssonLogarithmic}). It is possible to adopt the logarithmic
point of view to identify $\sT$ as the stack of aligned log structures
(cf \cite[8]{ExpDegPairs} and \cite{BorneVistoliParabolic}).
This allows a more combinatorial description of $\sT$ as a colimit of the $\cA^n$
by \'etale morphisms.

The Hopf algebra $\QSym$ of quasi-symmetric functions is well-studied object that
arises as a generalization of symmetric functions. As an algebra, it is commutative 
and graded, and it is free%
\footnote{I.e., isomorphic to a polynomial algebra.}
over $\bZ$ with finitely many generators in each degree,
though writing down explicit integral generators is not straightforward 
(cf \cite{HazewinkelExplicit}). The coalgebra structure is not
cocommutative. It is straightforward to see that $\QSym$ arises as a certain projective
limit of polynomial rings in the category of graded algebras.

We prove that $\QSym$ arises as the Chow ring of $\sT$. To be more precise, we
calculate the Chow ring of $\cX \times \sT$ for a smooth algebraic stack $\cX$ of finite
type over the base field, and show that the colimit construction of $\sT$ gives rise
to an isomorphism
\[
\Chowbullet(\cX \times \sT) \cong \Chowbullet(\cX) \otimes \QSym
\]
in Theorem \ref{thm.main}. We then show that there exists an \'etale, but
non-separated morphism $\sT \times \sT \to \sT$ which exhibits $\sT$
as a monoid object and induces the comultiplication of $\QSym$ on the level
of Chow cohomology. It can be shown that 

In section \ref{sn.applications}, we use the fact that $\sT$ has an interpretation
as an open substack of the moduli stack of $3$-pointed semistable curves
$\fM_{0,3}^{ss}$ to calculate the intersection rings of $\fM_{0,3}^{ss}$.
Moreover, the stack $\fM_{0,2}^{ss}$ can be seen as a non-rigid variant of $\sT$,
which allows us to compute its Chow ring, and the action of the natural
involution on this ring, as a corollary of our computations for $\sT$.

\subsubsection*{Notation}
Throughout this note, we let $k$ be an algebraically closed field of characteristic $0$.
There is no loss for the reader in assuming $k$ to be $\bC$ throughout. We let $\fC$
be the category of finite ordered sets and order-preserving injections, and we let $\fC_{\leq m}$ 
be the full subcategory consisting of sets with $m$ or less elements.

\subsubsection*{Acknowledgments}
I am indebted to my advisor Andrew Kresch for giving me the initial inspiration for this work
and helping me develop the ideas within. I would like to thank Johannes Schmitt,
Rahul Pandharipande, Dario de Stavola, and Julian Rosen for helpful comments. 
I am supported by Swiss National Science Foundation grant 156010.

\section{Quasisymmetric functions}
Quasisymmetric functions are a generalization of the well-known Hopf algebra of symmetric
functions due to Gessel (\cite{GesselQuasi}). Since we will identify the intersection ring
of the stack of expanded pairs as the ring of quasisymmetric functions, we first take some
time to describe this ring in more detail. All of the material in this section is classical, except
for possibly Proposition \ref{prop.qsym-limit}, which we did not find in the literature. A very
in-depth treatment of quasisymmetric functions can be found in \cite{QuasiSchur,GrinReinerHopf}.

\subsection{Definitions}
Given any totally ordered set $\fI$, we can consider the commutative graded ring
\begin{equation}
\bZ \llbracket \alpha_i \rrbracket_{i \in \fI}
\end{equation}
of formal power series in the variables $\{\alpha_i\}_{i \in \fI}$ and with $\bZ$-coefficients.

\begin{definition}
The algebra of \emph{quasisymmetric functions on the index set $\fI$} is the subring
$\QSym_\fI \subset \bZ \llbracket \alpha_i \rrbracket_{i \in \fI}$
consisting of power series $f$ of bounded degree satisfying the following condition:

For every two increasing sequences $i_1 < \dots < i_\ell$ and $j_1 < \dots < j_\ell$ of
elements of $\fI$ of length $\ell$, and for every $(I_1, \dots, I_\ell) \in \bZ^\ell_{>0}$,
the coefficients of $f$ for the monomials $\alpha_{i_1}^{I_1} \cdots \alpha_{i_\ell}^{I_\ell}$
and $\alpha_{j_1}^{I_1} \cdots \alpha_{j_\ell}^{I_\ell}$ are equal.

It is also possible to define quasisymmetric functions with $R$-coefficients for
any commutative ring $R$ as $\QSym_\fI \otimes R$.
\end{definition}
\begin{notation*}
We will usually consider the index set $\fI = \bZ_{>0}$, and write $\QSym := \QSym_{\bZ_{>0}}$.
We also define $\QSym_n := \QSym_{[n]}$, where $[n] := \{1 < \dots < n \}$.
\end{notation*}

\begin{definition}
Let $n\in \bN$. A \emph{composition} $I$ of $n$ (of length $\ell(I) := \ell$) is an ordered $\ell$-tuple 
$(I_1, \dots, I_\ell) \in \bZ_{>0}$ of positive integers such that $I_1 + \dots + I_\ell = n$.
Let $\Comp$ be the set of all compositions, and let $\Comp_n$ be the set of compositions of $n$.
Given two compositions $I=(I_1,\dots,I_k)$ and $J=(J_1,\dots, J_\ell)$, we write
\[
I \cdot J := (I_1, \dots, I_k, J_1, \dots, J_\ell).
\]
\end{definition}

There is a natural basis of $\QSym_{\fI}$ called the \emph{monomial basis}, indexed by compositions.
For a composition $I$, let $M_I$ be the monomial
\begin{equation}
M_I := \sum_{i_1 < \dots <i_\ell} \alpha_{i_1}^{I_1} \cdots \alpha_{i_\ell}^{I_\ell}.
\end{equation}
Then the monomials $M_I$ form a basis of $\QSym_\fI$ if $\fI$ is infinite, and the monomials indexed
by compositions $I$ with $\ell(I) \leq \vert \fI \vert$ form a basis of $\QSym_\fI$ if $\fI$ is finite%
\footnote{All $M_I$ for $\ell(I) > \vert \fI \vert$ are identically $0$.}. In particular, this implies that
$\QSym_\fI$ does not depend on the index set as long as $\fI$ is infinite. Nonetheless, keeping
track of indices can often be useful.

It is clear that $\deg(M_I) = n$ if $I$ is a composition of $n$, hence compositions of $n$ give rise to
a basis for the degree $n$ part of $\QSym$. The assignment $\fI \mapsto \QSym_\fI$ defines a
contravariant functor $\mathfrak{QSym}$ from the category of totally ordered sets and order-preserving injections to
the category of commutative graded rings; a morphism $g: \fI \incl \fJ$ maps to the "evaluation" homomorphism
setting all $\alpha_j$ for $j \notin g(\fI)$ to $0$.
This functor behaves nicely in the monomial basis:
\begin{proposition}
\label{prop.qsym.restriction}
For any order-preserving injection $g: \fI \incl \fJ$ and any composition $I$, we have
\[
\mathfrak{QSym}(g)(M_I) = M_I.
\]
In particular, if $\fI$ is finite, then the kernel of $\mathfrak{QSym}(g)$ is generated by monomials
$M_I$ for compositions $I$ satisfying $\ell(I) > \vert \fI \vert$, and the restriction of
$\mathfrak{QSym}(g)$ to the subalgebra generated by monomials $M_I$ for $\ell(I) \leq \vert \fI \vert$
is an isomorphism.\qed
\end{proposition}
Note that there also exists a covariant functor from the category of totally ordered sets
and order-preserving injections to the category of graded $\bZ$-modules, sending $M_I \to M_I$;
however, these inclusions of $\bZ$-modules are not algebra homomorphisms. For example,
for the inclusion $\{1\} \incl \{1 < 2\}$, the product $M_{(1)} \cdot M_{(1)}$ is equal to
$M_{(2)}$ in the source, but it is $M_{(2)} + 2 M_{(1,1)}$ in the target.

\begin{example*}
Let us list all monomials of degree up to $3$, with index set $\bZ_{>0}$.
\begin{itemize}
\item $M_\emptyset = 1$.
\item $M_{(1)} = \sum_i x_i = x_1 + x_2 + \dots$.
\item $M_{(2)} = \sum_i x_i^2 = x_1^2 + x_2^2 + \dots$.
\item $M_{(1,1)} = \sum_{i < j} x_i x_j = x_1 x_2 + x_1 x_3 + x_2 x_3 + \dots$.
\item $M_{(3)} = \sum_i x_i^3 = x_1^3 + x_2^3 + \dots$.
\item $M_{(2,1)} = \sum_{i < j} x_i^2 x_j = x_1^2 x_2 + x_1^2 x_3 + x_2^2 x_3 + \dots$.
\item $M_{(1,2)} = \sum_{i < j} x_i x_j^2 = x_1 x_2^2 + x_1 x_3^2 + x_2 x_3^2 + \dots$.
\item $M_{(1,1,1)} = \sum_{i < j < k} x_i x_j x_k = x_1 x_2 x_3 + x_1 x_2 x_4 + x_1 x_3 x_4 + x_2 x_3 x_4 + \dots$.
\end{itemize}
\end{example*}

\subsection{Free generators for the multiplication}
There is a description of the multiplication of monomials in purely combinatorial terms.

\begin{proposition}[{\cite[Prop 5.3]{GrinReinerHopf}}]
\label{prop.qsym.multiplication}
Fix an infinite index set $\fI$, and $\ell, m \in \bZ_{> 0}$. Fix two pairwise disjoint chain
posets $\{i_1 < \dots < i_\ell\}$ and $\{j_1 < \dots < j_m\}$. 

For compositions $I= (I_1, \dots, I_\ell)$ and
$J = (J_1, \dots, J_m)$, we have
\[
M_I \cdot M_J = \sum_f M_{\wt f},
\]
where the sum runs over all surjective, strictly order-preserving maps
\[
\{i_1,\dots,i_\ell\} \sqcup\{j_1, \dots, j_m\} \to \{1, \dots, n \}
\]
for some $n \in \bN$, and the composition $\wt f$ is defined by
\[
(\wt f)_x := \sum_{i_u \in f^{-1}(x)} I_u + \sum_{j_u \in f^{-1}(x)} J_u.
\]
\end{proposition}

As a consequence of Proposition \ref{prop.qsym.restriction} and Proposition \ref{prop.qsym.multiplication},
we can perform  additive computations involving only monomials of length $\leq \ell$ using any index set
with $\ell$ elements, and we can compute the product of two monomials of degree $\ell$ and $m$ using 
any index set with $\ell + m$ elements.

\begin{example*}
We can compute the product $M_{(1,2)} \cdot M_{(1,1)}$ as follows, using Proposition \ref{prop.qsym.multiplication}.
\begin{align*}
M_{(1,2)} \cdot M_{(1,1)} = &  M_{(1,2,1,1)}  + 2M_{(1,1,2,1)} + 3M_{(1,1,1,2)}  \\
 +  &   M_{(2,2,1)} + M_{(1,3,1)} + M_{(2,1,2)} + 2M_{(1,1,3)} + M_{(1,2,2)} \\
  + &   M_{(2,3)}.
\end{align*}
\end{example*}
It can be shown that quasisymmetric functions form a polynomial algebra, i.e., there is a subset generating $\QSym$ as
an algebra, without any relations. To write down a more precise statement, we need to introduce some notation.

\begin{definition}
Let $I,J\in \Comp$. We say that $I \leq J$ if
\begin{itemize}
\item either there exists an $i \leq \min\{\ell(I), \ell(J)\}$ such that $I_i < J_i$ and for every $j<i$, we have $I_j = J_j$,
\item or $I$ is a prefix of $J$, i.e. we can write $J=I\cdot K$ for another composition $K$.
\end{itemize}
This defines a total order on the set of compositions, sometimes called the \emph{lexicographic order},
because it is equal to the lexicographic order when two compositions have the same length.
\end{definition}
\begin{definition}
We say that a composition $I$ is a \emph{Lyndon composition} if every nonempty proper
suffix of $I$ is greater than $I$.
Concretely, this means that whenever we can write $I=J \cdot K$ for $J$ nonempty
compositions $J$ and $K$, we have $K>I$. Let $\fL$ be the set of Lyndon compositions.
\end{definition}

\begin{example*}
The empty composition is not Lyndon. Every composition of the form $(a)$ is Lyndon.
A composition of the form $(a,b)$ is Lyndon if and only if $b > a$.
A composition of the form $(a,b,c)$ is Lyndon if and only if $c > a$ and $b \geq a$.
\end{example*}

Let $\mu$ be the number-theoretic Möbius function, i.e. $\mu(d)$ is the sum of the $d$-th
primitive roots of unity.

\begin{proposition}[{\cite{TreueDarstellung}}]
The number $b_n$ of Lyndon compositions of length $n$ equals
\[
	b_n := \frac{1}{n} \sum_{d \vert n} \mu(d)\left( 2^{n/d} - 1 \right).
\]
Hence, the numbers $b_n$ satisfy $\sum_{d \vert n} db_d = 2^n -1$.
\end{proposition}

\begin{remark*}
The sequence $b_n$ starts as follows (sequence A059966 in OEIS):
\[
(b_1, b_2, \dots) = (1,1,2,3,6,9,18,\dots).
\]
\end{remark*}
\begin{theorem}[{\cite{HazewinkelInitial, HazewinkelExplicit}}]
$\QSym$ is isomorphic to a graded polynomial ring with $b_n$ generators in degree $n$.
\end{theorem}
We can write down a set of free rational generators quite explicitly. This is also possible over the
integers, but requires a larger notational effort\footnote{To see how an integral basis can be
constructed with the same index set, see \cite{HazewinkelExplicit},
or \cite[6.5]{GrinReinerHopf} for a more detailed explanation.}.
In fact, the morphism
\begin{align*}
	\bQ[x_I]_{I \in \fL} \to & \QSym \otimes \bQ \\
	x_I \mapsto & M_I
\end{align*}
is an isomorphism of graded algebras, for $\deg x_I = \vert I \vert$.

\begin{example*}
The generators corresponding to Lyndon compositions up to degree $4$ are as follows:
\[
M_{(1)}, M_{(2)}, M_{(3)}, M_{(1,2)}, M_{(4)}, M_{(1,3)}, M_{(1,1,2)}.
\]
\end{example*}

\subsection{Hopf algebra structure}
The comultiplication for quasisymmetric functions extend the one for symmetric functions.
In the monomial basis, it can be described as follows. Let
\begin{align*}
	\Delta\colon  \QSym &\to 	 \QSym \otimes \QSym & \varepsilon\colon  \QSym &\longrightarrow 	 \bZ \\
				 M_I &\mapsto	 \sum_{I=J \cdot K} M_J \otimes M_K &  n + \sum_{\vert I \vert \geq 1} M_I  &\mapsto n.
\end{align*}
and
\begin{align*}
	S\colon  \QSym &\to 	 \QSym \\
				 M_I &\mapsto  (-1)^{\ell(I)} \ \sum_{\mathclap{{\substack{J \text{ coarser than} \\ \rev(I)}}}} \  M_J,
\end{align*}
where the reverse $\rev(I)$ of a composition $I = (I_1, \dots, I_\ell)$ is $\rev(I) = (I_\ell, \dots, I_1)$
and $J$ is said to be \emph{coarser} than $K$ if $J$ can be obtained by successively summing
some of the adjacent entries of $K$.
\begin{proposition}
The triple $(\QSym, \Delta, \varepsilon)$ defines a coassociative coalgebra structure on $\QSym$ which
is compatible with the algebra structure. Moreover, the bialgebra $\QSym$ is a graded Hopf
algebra with antipode $S$.
\qed
\end{proposition}

\begin{example*}
Consider the monomial $M := M_{(3,1,4)}$. Then
\[
	\Delta(M) = M_{(3,1,4)} \otimes 1 + M_{(3,1)} \otimes M_{(4)} + M_{(3)} \otimes M_{(1,4)} + 1 \otimes M_{(3,1,4)}
\]
and
\[
	S(M) = - \left( M_{(4,1,3)} + M_{(5,3)} + M_{(4,4)} + M_{(8)}   \right).
\]
\end{example*}

\subsection{Quasisymmetric functions as a limit}

There is a construction of quasisymmetric functions as a categorical limit, which is useful for our 
purposes. Recall that $\fC$ denotes the category of finite ordered sets and order-preserving injections.
The assignment
\[
S \mapsto  \bZ[\alpha_s]_{s\in S} \qquad
(\phi: S \to T) \mapsto  g_\phi,
\]
where
\[
g_\phi (x_t) =
\begin{dcases} 
	\alpha_s, & 	\text{if there exist some } s \in S \text{ such that } \phi(s)=t, \\
	0, & 		\text{otherwise.}
\end{dcases}
\]
defines a contravariant functor from $\fC$ to the category of graded rings.

\begin{proposition}
\label{prop.qsym-limit}
For every finite ordered set $S$, consider the restriction morphism $\QSym \to \QSym_S$
from Proposition \ref{prop.qsym.restriction} for any order-preserving injection $S \incl \bZ$.
Then $\QSym$, together with this family of restriction morphisms, satisfies the universal
property of the colimit
\[
	\varprojlim_{\fC} \bZ[\alpha_s]_{s\in S}.
\]
\end{proposition}
\begin{proof}
For simplicity, we replace $\fC$ by the equivalent full subcategory consisting only of the objects
$[n] := \{ 1< \dots < n \}$ for $n\in \bN$. Note that the morphisms are generated by the family of morphisms
$j_i^n: [n] \to [n+1]$ for $i \in \{ 1, \dots, n+1 \}$, where $j_i^n$ is the unique map whose
image does not contain $i$.
Denote by
\[
\pi_n: \QSym \to \QSym_{[n]} \incl \bZ[\alpha_i]_{i=1, \dots, n}
\]
the projections, and
let $g^n_i := g_{j^n_i}$.
It is clear from the definition that the projections $\pi_n$ satisfy the necessary compatibility conditions,
since they preserve monomials.

To prove that $\QSym$ satisfies the universal property, let now $R^\bullet$ be any graded ring,
and assume we are given morphisms of graded rings
\[
f_n:R^\bullet\to \bZ[\alpha_1,\dots,\alpha_n]
\]
such that $g^n_i \circ f_{n+1} = f_n$, then for any $r\in R^d$ and every $1\leq i \leq n+1$, we have
\begin{equation}
\label{quasisym-relation}
f_{n+1}(r)(\alpha_1,\dots,\alpha_{i-1},0,\alpha_i,\dots,\alpha_n)=f_n(r)(\alpha_1,\dots,\alpha_n).
\end{equation}
This implies that for indices $i_1 < \dots < i_s$ and $j_1 < \dots < j_s$ and exponents
$k_1+ \dots + k_s = d$, the coefficient of $\alpha_{i_1}^{k_1} \cdots \alpha_{i_s}^{k_s}$
and the coefficient of $\alpha_{j_1}^{k_1} \cdots \alpha_{j_s}^{k_s}$ in
the homogeneous degree $d$ polynomial $f_{n+1}(r)$ agree.
This is proved by iterating the above equalities \eqref{quasisym-relation},
inserting zero everywhere except for indices in $\{i_1,\dots,i_s\}$, respectively
$\{j_1,\dots,j_s\}$. We conclude that $f_{n}(r)$ is a quasisymmetric function
of degree $d$ in $n$ variables, for every $n$.
Moreover, the value $f_e(r)$ for $e\geq 0$ is uniquely determined by $f_d(r)$,
because a quasisymmetric function of degree $d$ in any number of variables is determined
by its coefficients for monomials in the first $d$ variables, and because $f_0(r),\dots,f_{d-1}(r)$
are determined by \eqref{quasisym-relation}.
Hence, we can define $\psi(r)\in A_d$ by extending the quasisymmetric function $f_d(r)$
to a countable number of variables. By construction, we have $\pi_n \circ \psi = f_n$.
The lifting $\psi$ is unique, because the map $\pi_d|_{A_d}$ is injective.
\end{proof}

\begin{remark*}
With the same proof, we can also conclude that
\[
\varprojlim_{\fC_{\leq m}} \bZ[\alpha_s]_{s\in S} \cong \QSym_m.
\]
\end{remark*}

\section{Definitions and setup}
We denote the Chow groups of an algebraic stack $\cX$ by $\Chow_\bullet(\cX)$,
and the Chow ring of a smooth, equidimensional algebraic stack by
\[
\Chowbullet(\cX) = \Chow_{\dim(\cX) - \bullet} (\cX).
\]
with cohomological grading, assuming the latter is defined (for example,
if it has a stratification by quotient stacks).
All algebraic stacks are assumed to be defined over the field $k$.

\subsection{Chow groups for algebraic stacks admitting a good filtration}
In \cite{Kresch-cyc}, Chow groups, and the intersection ring for smooth algebraic
stacks, are only defined for an algebraic stack of finite type over $k$.
We will need to extend the definition to stacks which are close enough to
being of finite type for the purpose of intersection theory.
\begin{definition}
A \emph{good filtration by finite-type substacks} on an algebraic stack $\cX$,
which is assumed to be locally of finite type and of finite dimension, is
a collection $\{\cX_n\}_{n\in \bN}$ such that
\begin{itemize}
\item $\cX_n\subset \cX$ is an open substack of finite type;
\item $\cX_n \subset \cX_m$ for $n < m$;
\item $\dim (\cX \setminus \cX_n) < \dim X - n$.
\end{itemize}
In this situation, we will abbreviate to say that $\cX$ \emph{admits a good filtration}.
A morphism of algebraic stacks admitting a good filtration \emph{respects the filtration}
if it factors through the filtrations in the natural way, up to a degree shift.
\end{definition}

\begin{remark*}
The last condition could be weakened to only requiring that the $\cX_n$
jointly cover $\cX$, and that $\lim_{n \to \infty} \codim(\cup_{i=0}^n \cX_i) = \infty$;
however, all the stacks in our example already come with
a natural good filtration, and it is possible to pass from any filtration
by open substacks of finite type to a good filtration.
\end{remark*}

\begin{proposition}
There is a Chow group functor 
$\Chow_d(\cX) := \varprojlim_n \Chow_d(\cX_n)$
defined for stacks admitting a good filtration, satisfying the usual properties
as in \cite{Kresch-cyc},
with functoriality for morphisms of algebraic stacks which respect the filtration.
In fact, we have 
$\Chow_d(\cX) = \Chow_d(\cX_{\dim (\cX) + d}).$
If $\cX$ is smooth and admits a stratification by quotients stacks, there is also a ring
structure defined on the Chow groups, and we can compute the product
$\Chow^{d'}(\cX) \otimes \Chow^{d''}(\cX) \to \Chow^{d'+d''}(\cX)$
on $\cX_{d'+d''}$.
\end{proposition}
\begin{proof}
This is a straightforward consequence of excision.
\end{proof}

All of the algebraic stacks appearing in this note admit a good filtration,
and we will use the colimit above implicitly when discussing their Chow groups.

\subsection{\'Etale-local models}
We collect some facts about the stack quotients
$\cA^n := [\bA^n / \bG_m^n] = (\cA^1)^{\times n}$ for $n \in \bZ_{\geq 0}$
and their Chow rings.
Recall that $\cA^1$ is the moduli stack of pairs $(\cL, s)$
of a line bundle $\cL$ and a section $s$ of $\cL$, and hence $\cA^n$ is the
moduli stack of $n$-tuples of such pairs.

Since $\cA^n$ is a vector bundle of rank $n$ over
$B \bG_m^n \cong (B\bG_m)^{\times n}$,
the pullback of cycle classes from $B \bG_m^n$ is an isomorphism.
Let $[\bA^1 / \bG_m^n]_i$ correspond to the one-dimensional representation of
$B\bG_m^n$ given by $x \mapsto t_i x$.

\begin{proposition}
The graded ring $\Chowbullet (B\bG_m^n)$ is isomorphic to $\bZ[\alpha_1,\dots,\alpha_n]$,
where $\alpha_i$ corresponds to the class of the zero section of $[\bA^1 / \bG_m^n]_i$ under the
Gysin homomorphism
$\Chowbullet([\bA^1 / \bG_m^n]_i) \to \Chowbullet (B\bG_m^n)$.
\end{proposition}

\begin{corollary} We have
\[
\Chowbullet (\cA^n) \cong \bZ[\alpha_1,\dots,\alpha_n].
\]
\end{corollary}

\begin{lemma}\label{lem_NorBunBGm}
The normal bundle of $B\bG_m^n$ in $\cA^n$
has top Chern class $\alpha_1 \cdots \alpha_n$.
\end{lemma}
\begin{proof}
We can compute the normal bundle on any smooth atlas.
Choose the atlas $\bA^n$ to find that the normal bundle is $\cA^n$ itself,
seen as a vector bundle over $B\bG_m$. This a sum of line bundles
\[
N_{B\bG_m^n} \cA^n=\bigoplus_{i=1}^n [\bA^1 / \bG_m^n]_i.
\]
Hence we have
\[
c_n (N_{B\bG_m^n} \cA^n)=\prod_{i=1}^n \alpha_i.
\]
\end{proof}

\begin{definition}
We say that an algebraic stack $\fY$ over $k$ has the \emph{Chow K\"unneth property}
if for all algebraic stacks $\cX$ of finite type over $k$, the natural morphism
\begin{equation}
\label{eqn.chowkunneth}
\Chowbullet(\cX)\otimes \Chowbullet(\fY) \to \Chowbullet (\cX\times_k \fY)
\end{equation}
induced by functoriality is an isomorphism.
\end{definition}
\begin{remark*}
If $\fY$ has the Chow K\"unneth property, then \eqref{eqn.chowkunneth} will also be true
if $\cX$ only admits a good filtration.
\end{remark*}
\begin{lemma}
The classifying stack $B\bG_m$ has the Chow K\"unneth property.
\end{lemma}
\begin{proof}
First note that for any $n\geq1$, projective space $\mathbb{\bP}^{n-1}$ has the Chow
K\"{u}nneth property, by the formula for the Chow rings of projective bundles,
specialized to the case of a trivial bundle (cf. \cite[3.3]{fulton}). To prove the lemma,
note that $\cX \times [\bA^n/\bG_m]$ is a vector bundle over $\cX\times B\bG_m$,
which is $\cX\times\bP^{n-1}$ in codimension smaller than $n$.%
\footnote{In the topological setting, we have $BU(1) = \bC\bP^\infty$.
In this spirit, one could regard $B\bG_m$ as an algebraic version of $\bP^\infty$.}
By choosing $n$
high enough and applying the Chow K\"{u}nneth property of projective space,
we obtain the statement in any fixed degree.
\end{proof}
As a consequence, also $\cA^n$, for any $n$, has the Chow K\"{u}nneth property.
Occasionally, we will use coordinates indexed by a finite set $J$ instead of $\{1,\dots,n\}$.
In this situation, we will use the symbols $\bA^J$ and $\cA^J$. 

\subsection{The stack of configurations of line bundles}
We continue to use the notation $\cA^n$ from the last section. In this section, we
introduce the stack of expanded pairs $\sT$. The following treatment has been adapted
from \cite{ExpDegPairs}.
There are many equivalent moduli definitions (and universal families) for $\sT$; we will
describe one of them for the convenience of the reader, where $\sT$ appears in the
form of the moduli stack of \emph{configurations of line bundles}. This point of view
has been inspired from the approach to logarithmic structures in terms of line bundles
in \cite{BorneVistoliParabolic}.

\begin{definition}
A \emph{sheaf of totally ordered finite sets} is a constructible sheaf $E$ of partially ordered
nonempty\footnote{except over the empty set.} finite sets on the \'etale site of a scheme $S$ such that
any two sections are locally comparable. We identify $E$ with the stack on the \'etale site of $S$ whose objects
are sections of $E$, and where a unique morphism $x\to y$ exists if and only if $x \geq y$.
\end{definition}

We let $\mathfrak{Pic}$ be the category of line bundles.
The objects of $\mathfrak{Pic}$ over $S$ are line bundles on $S$,
and its morphisms are morphisms of line bundles. We should remark that morphisms are \emph{not}
required to be isomorphisms of line bundles, so $\mathfrak{Pic}$ is not a category fibered in groupoids,
as opposed to $B\bG_m$.

\begin{definition}\label{def.T-def}
Let $\sT$ be the stack whose objects are pairs $(E,L)$ of a sheaf of totally ordered sets $E$
and a morphism of stacks $L: E \to \mathfrak{Pic}$ such that
\begin{enumerate}
\item If $x\geq y$ are sections of $E$ and $L(x) \to L(y)$ is an isomorphism, then $x=y$.
\item $L(0) = \cO$, where $0$ is the unique section of $E$ that is minimal in all fibers.
\end{enumerate}
We call such an $L$ a \emph{diagram of line bundles} indexed by $E$.
We require that morphisms in $\sT$ are the identity on $L(0)$. We call $\sT$ the
stack of totally ordered configurations of line bundles.
\end{definition}

We recall the definition of the moduli stack of genus $g$ semistable curves with $n$
marked points $\fM_{g,n}^{ss} \subset\fM_{g,n}$,
whose objects are prestable curves of genus $g$ endowed with $n$ sections
such that the relative dualizing sheaf, twisted by the sections, has non-negative multidegree.
It is classical that $\fM_{g,n}^{ss}$ is an algebraic stack locally of finite presentation over $k$.

We will only need the genus $0$ case; for $n>0$, its geometric points are
nodal curves such that each component is isomorphic to projective space and such that each
component has at least two special points (either nodes or marked points).

\begin{theorem}[\cite{ExpDegPairs}]
There is an isomorphism of $\sT$ with the open substack of $\fM_{0,3}^{ss}$
where the last two points lie on the same component.
\end{theorem}

There is also a useful description of $\sT$ as a colimit. Given a finite, possibly empty,
ordered set $J$, there is a natural augmentation $\tilde{J}$, which is the union $J \sqcup \{ 0 \}$
of $I$ and a smallest element $0$. An order-preserving injection of finite sets $J \to K$
induces an open embedding $\cA^J \to \cA^K$. Assigning $J \mapsto \cA^J$,
together with the above morphisms, give rise to
diagrams of algebraic stacks indexed by $\cC_{\leq k}$, respectively $\cC$.

There are natural diagrams of line bundles on $\cA^J$ as follows. Consider the universal
family on $\cA^J$, that is, a collection $((\cL_i, s_i))_{i\in J}$ of line bundles and
sections.
Let $E$ be the quotient of the constant sheaf $\tilde{J}$ on $\cA^J$ by the relation
$i \sim i + 1$ on the locus where $s_{i+1}$ is nonzero.
Denote the elements of $J$ by $\{1 < \dots < n \}$ for $\vert J \vert = n$.
Then the following sequence of morphisms of line bundles defines a morphism $\tilde{J} \to \mathfrak{Pic}$
which descends to a diagram of line bundles on $\cA^J$ indexed by $E$.
\[
L_1^\vee \otimes \dots \otimes L_n^\vee \xrightarrow{s_n}
L_1^\vee \otimes \dots \otimes L_{n-1}^\vee \xrightarrow{s_{n-1}}
\dots \xrightarrow{s_2} L_1^\vee \xrightarrow{s_1} \cO
\]
This gives rise to morphisms $\cA^J \to \sT$ for every $J$,
compatible with the embeddings $\cA^J \incl \cA^K$.

\begin{proposition}[{\cite[Prop 8.3.1]{ExpDegPairs}}]
These morphisms are \'etale and induce an equivalence
\begin{equation}
\label{eq.injlimT}
\displaystyle\varinjlim_\cC \cA^J \xrightarrow{\sim}\sT.
\end{equation}
\end{proposition}

Setting
\[
\sT^{\leq d} := \varinjlim_{\cC_{\leq d}} \cA^J \incl \sT ,
\]
we see that the collection $\{\sT^{\leq d}\}_k$ forms a good filtration of $\sT$.
The following observation will be useful:
$\sT^{\leq d}$ possesses a unique closed point with stabilizer $B\bG_m^d$,
whose complement is $\sT^{\leq d-1}$. It pulls back to the image of the origin in
$\cA^J$ for $\vert J \vert = d$. We will sometimes call a geometric point of $\sT$ over this point
an \emph{accordion of length $d$}, according to the geometric picture in $\fM_{0,3}^{ss}$:
it is a nodal curve which is a chain of $d+1$ projective spaces with one marking
on one end two markings on the other end, and the action of the automorphism group
$\bG_m^d$ looks like an accordion being played.
\section{Cycle group calculations}
\subsection{\texorpdfstring{The Chow ring of $\sT$}{The Chow ring of T}}

We will prove the following result.
\begin{theorem}
\label{thm.main}
Let $\cX$ be an algebraic stack, smooth over $k$ and
admitting a good filtration, which has a stratification by quotient stacks.

\begin{enumerate}
\item
The natural morphism
\begin{equation}
\Chowbullet (\cX\times \sT^{\leq d}) =
\Chowbullet (\varinjlim_{\cC_{\leq d}}\cX\times\cA^n) \xrightarrow{}
\varprojlim_{\cC_{\leq d}} \Chowbullet (\cX\times\cA^n) =
\Chowbullet(\cX) \otimes \QSym_d
\end{equation}
induced (via functoriality with regard to \'{e}tale morphisms)
by the colimit \eqref{eq.injlimT} is an isomorphism for all $d$.
\item
The natural morphism
\begin{equation}
\Chowbullet (\cX\times\sT) =
\Chowbullet (\varinjlim_\cC \cX\times\cA^n) \xrightarrow{}
\varprojlim_\cC \Chowbullet (\cX\times\cA^n) =
\Chowbullet(\cX)\otimes\QSym
\end{equation}
is an isomorphism.
\end{enumerate}
In particular, the stack $\sT$ has the Chow K\"unneth property, and
$\Chowbullet (\sT) \cong \QSym.$
\end{theorem}
\begin{remark}
If $\cX$ admits a good filtration, but is not smooth, Theorem \ref{thm.main} is still true,
using almost the same proof, for the Chow groups without ring structure.
\end{remark}
\begin{proof}
First we note that the second part follows from the first. To see this, note
that the complement of $\sT^{\leq d}$ has codimension $d+1$, hence for any $d'\leq d$,
\[
\Chow^{d'} (\sT \times \cX) \cong \Chow^{d'} (\sT^{\leq d} \times \cX)
\]
by excision. To prove the first part, we fix notation by setting $B:=\Chowbullet(\cX)$.
Also, for any algebraic stack $\fY$, let
$\fY_\cX:=\cX\times\fY$,
and for morphisms $f:\fY\to\fZ$, let
$f_\cX:=\id\times f: \fY_\cX \to \fZ_\cX$.

We prove the statement by induction on $d$.
More precisely, we prove by induction that the morphism
$\Chowbullet (\sT^{\leq d}_\cX)\to \Chowbullet (\cA^d_\cX)$
is injective, with image equal to $B\otimes\QSym_d$.
The case $d=0$ is obvious. Suppose the theorem has been proven for $d-1$.
By Lemma \ref{lem_NorBunBGm}, the map
$\Chowbullet ((B\bG_m^d)_\cX) \to \Chow^{\bullet+d} (\cA^d_\cX)$
induced by inclusion is injective.
Since $(p_d)_\cX$ is \'{e}tale, the same holds for
$\Chowbullet ((B\bG_m^d)_\cX) \to \Chow^{\bullet+d} (\sT^{\leq d}_\cX)$.
We will determine its image. By excision, the following commutative diagram is exact.
\begin{equation}
\label{dia.2shortexact_2}
\begin{tikzcd}[column sep=small]
0 \arrow{r} & \Chowbullet ((B\bG_m^d)_\cX) \arrow{r} &
\Chow^{\bullet+d} (\cA^d_\cX) \arrow{r}&
\Chow^{\bullet+d} ((\cA^d \setminus B\bG_m^d)_\cX) \arrow{r} & 0 \\
0 \arrow{r} & \Chowbullet ((B\bG_m^d)_\cX) \arrow{r} \arrow{u}{(id_{B\bG_m^d})_\cX^*} &
\Chow^{\bullet+d} (\sT^{\leq d}_\cX) \arrow{r} \arrow{u}{(p_d)_\cX^*} &
\Chow^{\bullet+d} (\sT^{\leq d-1}_\cX) \arrow{r} \arrow{u}{(p_d|_{(\cA^d \setminus B\bG_m^d)})_\cX^*}& 0
\end{tikzcd}
\end{equation}
Let $\psi_i : \cA^{d-1}_\cX \to \cA^{d}_\cX$ induced by $j^{d-1}_i$.
Consider the following (not necessarily commutative) diagram.
\begin{equation}
\begin{tikzcd}
\label{dia_commut_2}
\Chowbullet ((\cA^d \setminus B\bG_m^d)_\cX) \arrow{dr}[swap]{\psi_i^*} \arrow{drr}{\psi_j^*} &{} &{} \\
{} & \Chowbullet (\cA^{d-1}_\cX) \arrow{r}{\sim} & \Chowbullet (\cA^{d-1}_\cX) \\
\Chowbullet (\sT^{\leq d-1}_\cX) \arrow{uu}{(p_d|_{(\cA^d \setminus B\bG_m^d)})_\cX^*} \arrow{ur}{(p_{d-1})_\cX^*} \arrow{urr}[swap]{(p_{d-1})_\cX^*} &{} &{}
\end{tikzcd}
\end{equation}
We take the horizontal map $h$ to be induced by the order-preserving bijection
$\{1,\dots, \hat{j}, \dots,d\} \to \{1,\dots, \hat{i}, \dots,d\}$.
Then, because $\sT^{d-1}_\cX$ is a colimit over $\cC_{\leq d-1}$,
\begin{equation}
\label{eq_compat_rel_2}
h \circ \psi_i^* \circ (p_d|_{(\cA^d \setminus B\bG_m^d)})_\cX^*
=(p_{d-1})_\cX^*=\psi_j^* \circ (p_d|_{(\cA^d \setminus B\bG_m^d)})_\cX^*.
\end{equation}
Since $(p_{d-1})_\cX^*$ is injective by induction hypothesis,
so is $(p_d|_{(\cA^d \setminus B\bG_m^d)})_\cX^*$.
By the five lemma, applied to \eqref{dia.2shortexact_2},
we deduce that $(p_d)_\cX^*$ is injective.
First, fill in the known (by the induction hypothesis) terms in \eqref{dia.2shortexact_2}.
\begin{equation}
\begin{tikzcd}[column sep=small]
0 \arrow{r} & B\otimes\bZ [\alpha_i]_{i=1}^d \arrow{r}{\cdot \alpha_1 \cdots \alpha_d} &
B\otimes\bZ [\alpha_i]_{i=1}^d \arrow{r}&
B\otimes\bZ [\alpha_i]_{i=1}^d/(\alpha_1 \cdots \alpha_d) \arrow{r} & 0 \\
0 \arrow{r} & B\otimes\bZ [\alpha_i]_{i=1}^d \arrow{r} \arrow{u}{id} &
\Chowbullet (\sT^{\leq d}_\cX) \arrow{r} \arrow{u}{(p_d)_\cX^*} &
B\otimes\QSym_{d-1} \arrow{r} \arrow{u}{(p_d|_{(\cA^d \setminus B\bG_m^d)})_\cX^*}& 0
\end{tikzcd}
\end{equation}
Likewise, we fill in the known groups in \eqref{dia_commut_2}.
\begin{equation}
\begin{tikzcd}
B\otimes\bZ [\alpha_i]_{i=1}^d/(\alpha_1 \cdots \alpha_d) \arrow{dr}[swap]{\alpha_i\mapsto 0} \arrow{drr}{\alpha_j\mapsto 0} &{} &{} \\
{} & B\otimes\bZ [\alpha_\ell]_{\ell=1,\ell\neq i}^n \arrow{r}{\sim} & B\otimes\bZ [\alpha_\ell]_{\ell=1,\ell\neq j}^n \\
B\otimes\QSym_{d-1} \arrow{uu}{(p_d|_{(\cA^d \setminus B\bG_m^d)})_\cX^*} \arrow{ur}{(p_{d-1})_\cX^*} \arrow{urr}[swap]{(p_{d-1})_\cX^*} &{} &{}
\end{tikzcd}
\end{equation}
To conclude, we apply the induction hypothesis and Lemma \ref{lem.vezzvist}.
The latter implies that the image of $\Chowbullet (\sT^{\leq d}_\cX)$ 
in $B\otimes \bZ[\alpha_1, \dots, \alpha_d]$ under $(p_d)_\cX^*$
can be identified with
\[
B \otimes \left(
\bZ[\alpha_1, \dots, \alpha_d]
\times_{\bZ[\alpha_1, \dots, \alpha_d] / (\alpha_1 \cdots \alpha_d)}
\QSym_{d-1} \right).
\]
By the induction hypothesis, the morphism $(p_{d-1})_\cX^*$ is the inclusion
of $B\otimes\QSym_{d-1}$ into a polynomial ring over $B$ with $d-1$ ordered variables.
Hence, by \eqref{eq_compat_rel_2}, the image of $B\otimes\QSym_{d-1}$ in
$B\otimes\bZ [\alpha_i]_{i=1}^d/(\alpha_1 \cdots \alpha_d)$
is exactly equal to (the image in $B\otimes\bZ [\alpha_i]_{i=1}^d/(\alpha_1 \cdots \alpha_d)$ of)
those $f\in B\otimes \QSym_d$ such that $f\notin (\alpha_1 \cdots \alpha_d)$.
By the commutativity of \eqref{dia.2shortexact_2}, this implies that the image of
$(p_d)_\cX^*$ contains a representative mod $(\alpha_1 \cdots \alpha_d)$
of every $f\in B\otimes \QSym_d$ such that $f\notin (\alpha_1 \cdots \alpha_d)$,
or to say it differently, it contains a representative mod $(\alpha_1 \cdots \alpha_d)$
of every quasi-symmetric polynomial with $B$-coefficients that has weight less than $d$.
Furthermore, it also contains $(\alpha_1, \dots, \alpha_d) \subset \QSym_d$ (from the
left-hand side of the fiber product).
Via a direct computation using the colimit description and functoriality,
it is immediate that any polynomial in the image must be quasisymmetric, and we have
shown that the image contains all quasisymmetric polynomials. This concludes the proof.
\end{proof}

\begin{lemma}[{\cite[Lemma 4.4]{VezzosiVistoliHigherAlgebraic}}, graded variant]
\label{lem.vezzvist}
Let $A$, $B$, and $C$ be graded rings, and let $f: B \to A$ and let $g: B \to C$ be morphisms.
Suppose that there exists a homomorphism of abelian groups $\phi: A \to B$ such that:
\begin{enumerate}
\item The sequence
\[
0 \to A \xrightarrow{\phi} B \xrightarrow{g} C \to 0
\]
is exact;
\item the composition $f \circ \phi: A \to A$ is the multiplication by an
element $a \in A$ of pure degree which is not a zero-divisor.
\end{enumerate}
Then $f$ and $g$ induce an isomorphism of graded rings
\[
(f, g): B \to A \times_{A/(a)} C,
\]
where $A \to A/(a)$ is the projection and $C \to A/(a)$ is induced by 
$C \cong B / \im (\phi) \to A / \im(f \circ \phi)=A/(a)$.\qed
\end{lemma}

\begin{remark}
The geometric meaning of most of the classes in $\Chowbullet(\sT)$ is
somewhat mysterious to the author. The subring of \emph{symmetric} functions
is the ring generated by closed substacks: it follows from the previous calculation that
the class of $\sT^{\geq d}$ is $M_I$, where
\[
I = \underbrace{(1, \dots, 1)}_{d \text{ times}}.
\]
In other words, it is the $d$-th elementary symmetric function.
On the other hand, consider a generator $L$ of the Picard group of $\cA^1$.
The unique (up to isomorphism) line bundle on $\sT$ whose restriction to the open substack
$\cA^1 \subset \sT$ is isomorphic to $L$ has first Chern class $\pm M_{(1)}$.
Hence, we readily obtain powers of $M_{(1)}$ as Chern classes of vector bundles, for
example
\[
M_{(1)}^2 = M_{(2)} + 2 M_{(1,1)}.
\]
However, it is not clear how we can naturally produce non-symmetric classes
such as $M_{(1,2)}$.
\end{remark}

\subsection{Hopf algebra structure}
We construct a morphism $\mu: \sT \times \sT \to \sT$ that
gives rise to the comultiplication of the Hopf algebra. Given two pairs
$(E_1, L_1)$ of $(E_2, L_2)$ of a sheaf of totally ordered sets and a diagram of
line bundles indexed by that sheaf, we define $E$ as the sheaf of partially ordered
sets that arises as the sheafification of the presheaf that identifies the
unique minimal element of $E_2$ with the highest element of $E_1$. Then
\[
L: E \to \mathfrak{Pic}
\]
is the morphism which takes a section $x$ to $L_1(x)$ if $x$ is a section of $E_1$,
and to $L_2(x) \otimes \tilde{L}$ if $x$ is a section of $E_2$, where $\tilde{L}$ is
the image under $L_1$ of the local maximum of $E_1$.

Informally, we associate to two sequences of line bundles and morphisms
\[
L_{1,n_1} \to \dots \to L_{1,1} \to \cO \text{ and }
L_{2,n_2} \to \dots \to L_{2,1} \to \cO
\]
the sequence
\begin{align*}
L_{2,n_2} \otimes L_{1,n_1} \to L_{2,n_2-1} \otimes L_{1,n_1} \to
\dots \to L_{2,1} \otimes L_{1,n_1} \\
\to L_{1,n_1} \to L_{1,n_1 -1 } \to \dots \to L_{1,1} \to \cO.
\end{align*}

For $n = n_1 + n_2$, we can identify
$\Psi_{n_1, n_2}: \cA^{n_1} \times \cA^{n_2} \xrightarrow{\simeq} \cA^n$
under the obvious isomorphism preserving the order of the coordinates.
Since the former form an \'etale cover of $\sT \times \sT$ and the latter form an \'etale cover
of $\sT$, it will be important so understand how they interact.
\begin{proposition}
Let $n, n_1, n_2, \Psi_{n_1, n_2}$ as above.
There is an isomorphism
\[
\mu \circ (p_{n_1} \times p_{n_2}) \simeq p_n \circ \Psi_{n_1, n_2}.
\]
\end{proposition}
\begin{proof}
This is clear by the definition of $\mu$.
\end{proof}

\begin{lemma}
The morphism $\mu$ defined by the previous construction is representable \'etale.
\end{lemma}
\begin{proof}
Since the stabilizer groups of an accordion of length $d$ is $\bG_m^d$ and the
morphism maps a pair of accordions of length $d_1$ and $d_2$ to an accordion
of length $d_1+d_2$, the morphism is stabilizer-preserving, hence representable.
For representable morphisms, we can check the property of being \'etale on an atlas
for the target. Since the property of being \'etale is \'etale local on the source,
it is enough to show that
\[
\cA^{n_1} \times \cA^{n_2} \to \sT^{\leq n}
\]
is \'etale whenever $n_1 + n_2 \leq n$, and for that it is enough to show that it is
\'etale in the case $n_1 + n_2 = n$. After base change by $\cA^n \to \sT^{\leq n}$,
we have to show that the $2$-fiber product
\[
(\cA^{n_1} \times \cA^{n_2}) \times_\sT \cA^n
\]
is \'etale over $\cA^n$, but this is just the diagonal of $\cA^n \to \sT$.
\end{proof}
\begin{remark}
One should note that while $\mu$ is \'etale and quasi-finite, it is not separated, since it
admits sections. For example, we have the embeddings
$\sT \cong \sT \times \sT^{\leq 0} \incl \sT \times \sT$ and
$\sT \cong \sT^{\leq 0} \times \sT \incl \sT \times \sT$,
both of which are sections of $\mu$.
\end{remark}
\begin{remark}
As suggested by the notation, the morphism $\mu$ is the multiplication morphism
exhibiting $(\sT, \mu, \Spec (k) \incl \sT) $ as a monoid object in the $2$-category
of algebraic stacks.
\end{remark}

\begin{theorem}
The comultiplication $\Delta: \QSym \to \QSym \otimes \QSym$ is equal to the pullback of cycle classes $\mu^*$.
\end{theorem}

\begin{proof}
It is enough to show that $\mu^* (M_I) = \Delta (M_I)$ for every composition $I$, where
$M_I$ is the monomial basis element indexed by $I$.
Fix $I$ and let $n = \lvert I \rvert$ be the size of $I$, which is also the degree of $M_I$.
We have established in the proof of Theorem \ref{thm.main} that the pullback of cycle classes
on $\sT$ of cohomological degree at most $n$ to $\cA^n$ is injective. The Chow ring of
$\sT \times \sT$ is naturally bigraded as a tensor product of graded algebras. Using the
same technique twice, we can see that for cycle classes of bidegree at most $(n_1, n_2)$,
the pullback to $\cA^{n_1} \times \cA^{n_2}$ is injective.

Note that the complement of the union of the images of $\cA^{n_1} \times \cA^{n_2}$
as we range over pairs $(n_1,n_2)$ with $n_1+n_2=n$ has codimension $n+1$.
This implies that for any $d \leq n$,
a class in $\Chow^d(\sT \times \sT)$ can be uniquely identified by its image in
\[
\prod_{n_1 + n_2 = n} \Chow^d(\cA^{n_1} \times \cA^{n_2})
\]
by the product of the various pullbacks, by excision(\cite[Thm. 2.1.12(iv-v)]{Kresch-cyc}).
It remains to be shown that the pullback of
$\mu^* M_I$ to $\Chow^n(\cA^{n_1} \times \cA^{n_2})$
is equal to the pullback of
\begin{equation}
\label{eq.phi-comult-formula}
\sum_{s=0}^\ell M_{(I_1, \dots, I_s)} \otimes M_{(I_{s+1}, \dots, I_\ell)}.
\end{equation}
Let $\Psi := \Psi_{n_1, n_2}$.
We use the fact that $\mu \circ (p_{n_1} \times p_{n_2})$ is isomorphic to $p_n \circ \Psi$.

Under this isomorphism, the monomial $M_I$ corresponding to a composition $I= (I_1, \dots, I_\ell)$
pulls back to the same monomial in $\QSym \otimes \QSym$, except that we regard the first $n_1$
variables as coming from the left-hand side of the tensor product,
and the remaining $n_2$ variables as coming from the right-hand side.

To make this precise, let us denote the variables in
$\Chowbullet(\cA^n), \Chowbullet(\cA^{n_1})$, and $\Chowbullet(\cA^{n_2})$ by
$\gamma_i, \alpha_i$, and $\beta_i$ respectively,
such that $\Psi^* \gamma_i = \alpha_i$ for $i \in \{1, \dots, n_1\}$ and
$\Psi^* \gamma_{n_1 + i} = \beta_i$
for $i \in \{1, \dots, n_2\}$.
Then we see that
\begin{align}
\Psi^* M_I & = \Psi^* \left( \sum_{i_1 < \dots < i_\ell} \gamma_{i_1}^{I_1} \cdots \gamma_{i_\ell}^{I_\ell} \right) \nonumber\\
& = \sum_{i_1 < \dots < i_\ell} \Psi^*(\gamma_{i_1})^{I_1} \cdots \Psi^*(\gamma_{i_\ell})^{I_\ell} \nonumber\\
& = \sum_{\substack{i_1 < \dots < i_k \leq n_1\\ n_1 < i_{k+1} < \dots < i_\ell}}
\Psi^*(\gamma_{i_1})^{I_1} \cdots \Psi^*(\gamma_{i_k})^{I_k}
\Psi^*(\gamma_{i_{k+1}})^{I_{k+1}}\cdots \Psi^*(\gamma_{i_\ell})^{I_\ell} \nonumber\\
& = \sum_{\substack{i_1 < \dots < i_k \leq n_1\\ n_1 < i_{k+1} < \dots < i_\ell}} (\alpha_{i_1} \otimes 1)^{I_1} \cdots (\alpha_{i_k} \otimes 1)^{I_k} (1 \otimes\beta_{(i_{k+1}-n_1)})^{I_{k+1}}\cdots (1 \otimes \beta_{(i_\ell - n_1)})^{I_\ell} \nonumber\\
& = \sum_{\substack{I = J \cdot K \\ J \in \Comp_{n_1} \\ K \in \Comp_{n_2}}} M_J \otimes M_K
= \sum_{\substack{I = J \cdot K \\ \ell(J) \leq n_1 \\ \ell(K) \leq n_2}} M_J \otimes M_K.
\label{eq.phi-comult-result}
\end{align}
On the other hand, the pullback of $p(\alpha_i)\otimes q(\beta_i) \in \QSym \otimes \QSym$
by $(p_{n_1} \times p_{n_2})$ is the evaluation homomorphism that sets all $\alpha_i$ to $0$
for $i > n_1$ and all $\beta_i$ to $0$ for $i > n_2$, and is the identity on the remaining variables.
Hence
\[
(p_{n_1} \times p_{n_2})^* M_J \otimes M_K =
\begin{dcases}
	M_J \otimes M_K, & \text{if } \ell(J) \leq n_1 \text{ and } \ell(K) \leq n_2 \\
	0 & \text{otherwise.}
\end{dcases}
\]
Hence, by applying the previous equation to \eqref{eq.phi-comult-formula}, we see that
\[
(p_{n_1} \times p_{n_2})^* (\Delta (M_I))
\]
is equal to the sum \eqref{eq.phi-comult-result}.
\end{proof}

\begin{remark*}
Since the structure morphism $\sT \to \Spec k$ and the inclusion of the generic point
$\Spec k \to \sT$ induce the unit and counit of $\QSym$, respectively, one could conjecture
that there could also be a morphism $\widehat{S}: \sT \to \sT$ inducing the antipode.
However, we have not been able to find a good candidate; among the reasons for that
is that if $\widehat{S}$ exists, it cannot be representable and in particular it will not be
an isomorphism, contrary to the situation for Hopf algebras. Since $\sT$ is only a
monoid, not a group, the existence of a geometric antipode seems unlikely.
\end{remark*} 

\subsection{A natural involution}\label{subsec.involution}
The stack $\sT$ comes equipped with an involution $\sigma$, which sends a
diagram of line bundles
\[
L_n \to L_{n-1} \to \dots \to L_1 \to \cO
\]
to the "dualized" diagram
\[
L_n \xrightarrow{s_1^\vee \otimes \id} L_1^\vee \otimes L_n \xrightarrow{} \dots
\xrightarrow{} L_{n-1}^\vee \otimes L_n \xrightarrow{s_n^\vee \otimes \id}
L_{n}^\vee \otimes L_n \cong \cO.
\]
This corresponds to the involution $\rev \colon \fC \to \fC$ of the index categeory $\fC$
which sends a finite ordered set to the same set with the opposite order, and
similarly for morphisms. There is an explicit description of $\sigma^*$ as follows.
\begin{proposition}\label{prop.involution}
The homomorphism of graded algebras $\sigma^*: \QSym \to \QSym$
is characterized in the monomial basis by
\[
\sigma^* M_J = M_{\rev(J)}.
\]
This is the unique nontrivial graded algebra automorphism of $\QSym$ which preserves
the monomial basis.
\end{proposition}
\begin{proof}
It is enough to show that the restriction $\sigma \vert_{\sT^{\leq d}}$ induces 
the claimed automorphism for each $d$.
Let $J = [d]$, and consider the \'etale covers $p\colon \cA^J \to \sT^{\leq d}$ and
$\tilde{p} \colon \cA^{\rev(J)} \to \sT^{\leq d}$.
The involution $\sigma \vert_{\sT^{\leq d}}$ does not lift to a morphism
$\cA^J \to \cA^J$ over $ \sT^{\leq d}$,
but it lifts to the identity $\id\colon\cA^J \to \cA^{\rev(J)}$, which is not induced
by any morphism in $\fC$.
Let $J=(J_1, \dots, J_\ell)$ be a composition of length at most $d$.
The monomial $M_J \in \Chow (\sT^{\leq d})$
pulls back to 
\[
\sum_{i_1 < \dots < i_\ell} \alpha_{i_1}^{J_1} \cdots \alpha_{i_\ell}^{J_\ell}
\]
in $\Chow(\cA^J)$ respectively $\Chow(\cA^{\rev(J)})$, where the indices $i_j$
are taken from the same set $[n]$, but the order relation is inverted.
Hence we see that $\tilde{p}^* \sigma^* M_J = p^* M_{\rev (J)}$.
For the second part, see \cite{QuasiRigidity}.
\end{proof}
Note that there does not exist any nontrivial algebra automorphism of $\QSym$
which preserves the monomial basis \emph{and} respects the comultiplication.

\section{\texorpdfstring{Application to $\fM_{0,2}^{ss}\text{ and } \fM_{0,3}^{ss}$}{Application to M(0,2)-ss and M(0,3)-ss}}
\label{sn.applications}
We can now leverage our calculations and the interpretation of $\sT$ as a substack of
$\fM_{0,3}^{ss}$ to obtain some immediate corollaries for the Chow groups of 
$\fM_{0,n}^{ss}$ for $n \in \{2,3\}$.

Consider the stack $\fM_{0,3}^{ss}$ of semistable curves of genus $0$ with $3$ marked points.
Note that any such curve over $k$ has precisely one stable component.
We can define a morphism $\Psi: \sT \times \sT \times \sT \to \fM_{0,3}^{ss}$ as follows:
Given three curves $C_1,C_2,C_3$ with markings $(\sigma_1,\dots,\sigma_9)$
such that $2$ markings on each curve lie on the same component, glue the stable components
along the three special points (in such a way that one marked points in $C_1$ is glued to the node
in $C_2$, and the other marked points to the node in $C_3$,
and similarly for marked points on the other curves).
The result is a nodal semistable curve with $3$ marked points and precisely one stable component.
A triple of morphisms in $\sT$, i.e. Cartesian diagrams, is taken to the glued Cartesian diagram.
This is an isomorphism, because it is a representable \'etale morphism whose
geometric fibers consist of single points.
Theorem \ref{thm.main} then implies the following.
\begin{corollary}
We have $\Chow^\bullet (\fM_{0,3}^{ss}) \cong \QSym \otimes \QSym \otimes \QSym$.
\end{corollary}
There is a variant of the stack of expansions, called the stack of non-rigid
expansions, which can be shown to be isomorphic to $\fM_{0,2}^{ss}$.
The two variants are connected by a rather simple relationship: moving
one of the marked points on the stable component of $\sT$, viewed as 
a stack of semistable curves, defines a $\bG_m$-action, and $\fM_{0,2}^{ss}$
is the quotient by this action. More precisely, there is the following result:
\begin{proposition}[{\cite[Prop. 3.3.4]{ExpDegPairs}}]\label{prop.m02-isom}
There is a canonical (up to a twist by inversion) isomorphism
$\fM_{0,2}^{ss} \cong B\bG_m \times \sT$,
exhibiting $\sT$ as the rigidification of $\fM_{0,2}^{ss}$ by the
normal subgroup $\bG_m$ of its inertia stack.
\end{proposition}
\begin{corollary}
We have $\Chow^\bullet (\fM_{0,2}^{ss}) \cong \QSym[\beta]$,
where $\beta$ is the pullback of a generator under the morphism
$\fM_{0,2}^{ss} \to B\bG_m$ corresponding to the $\bG_m$-action on $\sT$.
\end{corollary}

In fact, we can find a very explicit description of the isomorphism in Proposition
\ref{prop.m02-isom} as follows: the stack $\fM_{0,2}^{ss}$ is isomorphic to the 
stack $\sT^\mathrm{nr}$, whose definition is the same as Definition \ref{def.T-def},
except that we do not require condition (ii).\footnote{The symbol "nr" stands for "non-rigid".}

Then the isomorphism $\Phi: \sT^\mathrm{nr} \to \bG_m \times \sT$ takes 
a diagram
\[
L_n \to \dots \to L_1 \to L_0
\]
to the pair
\[
(L_0, L_n \otimes L_0^\vee \to \dots \to L_1 \otimes L_0^\vee \to \cO).
\]

There is a natural involution $\tau:\fM_{0,2}^{ss} \to \fM_{0,2}^{ss}$ given by exchanging
the two marked points. In the interpretation $\fM_{0,2}^{ss} \cong \sT^\mathrm{nr}$,
this corresponds to dualizing:
\[
(L_n \to \dots \to L_0) \mapsto (L_0^\vee \to \dots \to L_n^\vee).
\]
\begin{proposition}\label{prop.s2-action}
The action of $\tau$ on the Chow ring $\QSym[\beta]$ is given by mapping a
quasisymmetric function $g$ to $\sigma^* (g)$ for $\sigma$ as in \ref{subsec.involution},
and by mapping
\[
\beta \mapsto -\beta + M_{(1)} = -\beta + \sum_i \alpha_i.
\] 
\end{proposition}
\begin{proof}
First observe that $\tau$ restricts to $\sigma$ on the second factor, so we only need
to compute $\tau^* \beta$. Since $\tau^*$ is an automorphism of graded rings, it preserves
the grading of $\tau$, so it is enough to compute it on
$\bG_m \times \sT^{\leq 1} \cong \bG_m \times \cA^1,$
where we can see that the dual diagram $(L_0^\vee \to L_1^\vee)$
is sent by $\Phi$ to the pair
\[
(L_1^\vee, L_0^\vee \otimes L_1 \to \cO) =
(L_0^\vee \otimes (L_0^\vee \otimes L_1)^\vee, L_0^\vee \otimes L_1 \to \cO).
\]
We conclude by remarking that $c_1((L_0^\vee \otimes L_1)^\vee) = M_{(1)}$.
\end{proof}

\bibliographystyle{halpha}
\bibliography{aligned}

\end{document}